\documentclass[12pt,twoside]{amsart}

\usepackage{hyperref,mathtools}
\usepackage{amsthm, amsmath, amscd, amssymb,centernot}
\usepackage{amsfonts}
\usepackage[all]{xy}
\usepackage{palatino,eulervm}
\usepackage[T1]{fontenc}
\usepackage{xcolor}
\usepackage[left=2.5cm,top=2.5cm,bottom=3cm,right=2.5cm]{geometry}
\usepackage{tikz-cd}

\DeclareMathOperator{\rank}{rank}

\newcommand{\complex}{\mathbb C}

\theoremstyle{plain}
\numberwithin{equation}{section}
\newtheorem{theorem}{Theorem}[section]
\newtheorem*{theorem*}{Theorem}
\newtheorem{proposition}[theorem]{Proposition}
\newtheorem{lemma}[theorem]{Lemma}
\newtheorem{corollary}[theorem]{Corollary}

\theoremstyle{definition}
\newtheorem{question}[theorem]{Question}

\newtheorem{remark}[theorem]{Remark}
\newtheorem{example}[theorem]{Example}

\newtheorem{assumption}[theorem]{Assumption}

\begin{document}
\title{Seshadri constants on some flag bundles}
\author[Krishna Hanumanthu]{Krishna Hanumanthu}
\address{Chennai Mathematical Institute, H1 SIPCOT IT Park, Siruseri, Kelambakkam 603103, India}
\email{krishna@cmi.ac.in}

\author[Jagadish Pine]{Jagadish Pine}
\address{Chennai Mathematical Institute, H1 SIPCOT IT Park, Siruseri, Kelambakkam 603103, India}
\email{pine@cmi.ac.in}

\thanks{Both authors were partially supported by a grant from Infosys
  Foundation. }

\date{28 March, 2023}
\maketitle
\begin{abstract}
Let $X$ be a smooth complex projective curve $X$ and let $E$ be a vector bundle on $X$ which is 
not semistable.  We consider a flag bundle $\pi: \text{Fl}(E) \to X$ parametrizing certain flags of 
 fibers of $E$. The dimensions of the successive quotients of the flags 
 are determined by the ranks of vector bundles appearing
 in the Harder-Narasimhan filtration of $E$.  We compute the Seshadri constants of nef line bundles on 
 $\text{Fl}(E)$.
\end{abstract}

\section{Introduction}\label{intro}

Seshadri constants were introduced by Demailly \cite{Dem} in 1990 as a tool to study jet separation of line bundles on complex projective varieties. They have become an important topic in the study of positivity in algebraic geometry. 

We quickly recall their definition. Let $L$ be a nef line bundle on a projective variety $X$. 
For a point $x \in X$, The \textit{Seshadri constant} of $L$ at $x$ is defined as
$$\varepsilon(X,L,x):= \inf\limits_{\substack{x \in C}} \frac{L\cdot
C}{{\rm mult}_{x}C} \;\: ,$$
where the infimum is taken over all irreducible and reduced curves $C \subset X$ passing through $x$. Here 
$L\cdot C$ denotes the intersection multiplicity, and ${\rm mult}_x C$ denotes the 
multiplicity of the curve $C$ at $x$. So $\varepsilon(X,L,x)$ depends only the numerical equivalence class of $L$. If the variety is clear from the context, we denote the Seshadri constant simply by $\varepsilon(L,x)$.

By Seshadri's criterion for ampleness, the line bundle $L$ is ample if and only if $\varepsilon(L,x) > 0$ for all $x \in X$. So Seshadri constants of ample line bundles are positive real numbers. 

In order to understand the behaviour of $\varepsilon(L,x)$ as $x$ varies, one defines the following numbers: 
$$\varepsilon(L,1) :=\sup\limits_{x\in X}
\varepsilon(L,x),$$  and 
$$\varepsilon(L) :=\, \inf\limits_{\substack{x \in X}} \varepsilon(L,x)\, .$$

It is known that the values of both $\varepsilon(L,1) $ and $\varepsilon(L)$ are achieved  at specific points. In fact, 
$\varepsilon(L,1)  = \varepsilon(L,x)$ for a \textit{very general} point $x \in X$.

It is easily seen that the following inequalities hold for any point $x \in X$. The dimension of $X$ is denoted by $n$. 
$$ 0 <  \varepsilon(L)\le  \varepsilon(L,x) \le \varepsilon(L,1)  \le \sqrt[n]{L^n}.$$

One is interested in computing the precise value of Seshadri constants, or at least in giving some bounds. 
In general, the larger the Seshadri constant is, the more positive $L$ will be. For example, if $L$ is very ample (in fact, if it is ample and base point free) then $\varepsilon(L,x) \ge 1$ for all $x\in X$. On the other hand, it may happen that $\varepsilon(L,x) < 1$ for some $x$ if the ample line bundle $L$ is not base point free.

Note that the above definition of Seshadri constants is meaningful only when the dimension of $X$ is two or more. Seshadri constants have been extensively studied on surfaces and the general picture is understood well in this case; see \cite{B99,H,BS} for a sample.  
In dimension at least three, a lot is known in specific cases such as abelian varieties and Fano varieties (see \cite{B,L} for example), but the general picture is not well-understood. For a detailed survey, see \cite[Chapter 5]{Laz} and 
\cite{primer}.

In this paper, we compute Seshadri constants of ample line bundles on some flag bundles over complex projective curves. We now recall the construction of the flag bundles that we will study in this paper. 

Let $E$ be a vector bundle on a connected smooth projective curve $X$ defined over $\mathbb{C}$. We will assume that $E$ is not semistable. Let $$0=E_0 \subset E_1 \subset E_2 \subset \cdots \subset E_{d-1} \subset E_d=E$$ be the Harder-Narasimhan filtration of $E$. For every $1 \le j \le d-1$, let $r_j$ denote the rank of $E/E_j$. Then we have $r_1 > r_2 > \cdots >r_{d-1}$. We consider the flag bundle 
\begin{eqnarray}\label{flag-define}
\pi: \text{Flag}(r_{k_1}, r_{k_2}, \cdots , r_{k_{\gamma}}, E) \rightarrow X,
\end{eqnarray}
where $1 \le k_1 < k_2 < \cdots <k_{\gamma} \le d-1$. A point in a fiber over a point $x \in X$ represents successive quotients of the vector space $E_x$ of the following form

\begin{center}
	$E_x \rightarrow V_{k_1} \rightarrow  V_{k_2} \rightarrow \cdots \rightarrow V_{k_{\gamma}},$
\end{center} 
where $\text{dim}(V_{k_i})=r_{k_i}$ for $1 \le i \le \gamma$. We will denote this flag bundle \eqref{flag-define} simply by $\text{Fl}(E)$, since the positive integers $r_{k_1}, r_{k_2}, \cdots , r_{k_{\gamma}}$ are fixed throughout the article. 

The nef cone of the variety $\text{Fl}(E)$ was computed by Biswas and Parameswaran \cite{BP}. See Section \ref{prelims} for a brief description of their results. Using their results, one can explicitly describe the ample line bundles on $\text{Fl}(E)$. In order to compute  the Seshadri constants, we will first describe the cone of effective curves of $\text{Fl}(E)$ in Proposition \ref{curves-cone}. 
This enables us to compute the Seshadri constants. 

In \cite{BHNN}, the authors computed the Seshadri constants of ample line bundles on Grassmann bundles. Our approach to the flag bundle case is motivated by this paper.

In Theorem \ref{main-thm}, we give lower and upper bounds for the Seshadri constants of a nef line bundle $L$ on $\text{Fl}(E)$ in terms of the non-negative 
integers expressing $L$ as a linear combination of the generators of the nef cone of $\text{Fl}(E)$. We then show that 
the lower bound given in Theorem \ref{main-thm} is always achieved at specific points in $\text{Fl}(E)$. We also show that 
under some additional assumptions on the Harder-Narasimhan filtration of $E$, the upper bound 
given in Theorem \ref{main-thm}
is achieved at general points of $\text{Fl}(E)$.
We give some examples in Section \ref{examples}. 

We work throughout over the field $\mathbb{C}$ of complex numbers. The field of real numbers is denoted by $\mathbb{R}$.

\section{Nef cone of $\text{F\lowercase{l}}(E)$}\label{prelims}
In this section we recall the description of the nef cone of $\text{Fl}(E)$ given in \cite{BP}. 

Let $X$ be a nonsingular projective variety. The \textit{N\'eron-Severi group} of $X$ is defined to be the quotient group $$\text{NS}(X)= \text{Div}(X)/ \text{Num}(X),$$ where $\text{Div}(X)$ is the group of divisors on $X$ and $\text{Num}(X)$ is the subgroup of numerically trivial divisors. Then $\text{NS}(X)$ is a finitely generated abelian group. 

The \textit{N\'eron-Severi space} of $X$ is defined to be
$$\text{NS}(X)_{\mathbb{R}} := \text{NS}(X) \otimes_{\mathbb{Z}} \mathbb{R}.$$

The cone in $\text{NS}(X)_{\mathbb{R}}$ generated by all the real nef divisors is called the \textit{nef cone} of $X$. The cone generated by all the ample divisors is called the \textit{ample cone} of $X$.
The \textit{pseudo-effective cone} of $X$ is the closed cone in $\text{NS}(X)_{\mathbb{R}}$ generated by all the effective classes in $\text{NS}(X)_{\mathbb{R}}$. An element of the pseudo-effective cone is called a pseudo-effective divisor. 

A nef divisor is a limit of ample divisors and a pseudo-effective divisor is a limit of effective classes. In particular, a nef divisor is pseudo-effective. 

We will also be interested in the cone of curves in $X$. Let $N_1(X)_{\mathbb{R}}$ denote the vector space of real one-cycles on $X$ modulo numerical equivalence. There is a perfect pairing between $\text{NS}(X)_{\mathbb{R}}$ and 
 $N_1(X)_{\mathbb{R}}$  given by the intersection pairing on $X$. 
 
 The \textit{closed cone of curves} on $X$, denoted $\overline{\text{NE}}(X)$, is the closure of the cone spanned by all
 effective one-cycles on $X$. The nef cone and $\overline{\text{NE}}(X)$ are dual to each other under the intersection pairing. 
 
For more details, see \cite[Section 1.4.C]{Laz}.

We are interested in these cones for the flag bundle $\text{Fl}(E)$. Its nef cone was described by \cite{BP}.

For every $1 \le i \le \gamma$, 
let $$f_i:\text{Gr}_{r_{k_i}}(E) \rightarrow X$$ be the Grassmannian bundle of rank $r_{k_i}$ quotients of the vector bundle $E$. Let $\mathcal{L}_i$ be the pullback $f_i^*(\mathcal{L}')$, where $\mathcal{L'}$ is a line bundle on $X$ of degree 1.

Consider the following sequence of maps.

\begin{equation}
	\text{Fl}(E) \xhookrightarrow{\Phi} \prod_{i=1}^{\gamma} \text{Gr}_{r_{k_i}}(E) \hookrightarrow \prod_{i=1}^{\gamma} \mathbb{P}((\wedge^{r_{k_i}} E)^{*})
\end{equation}

The first embedding $\Phi$  is defined by sending a point  $$Q=(E_x \rightarrow V_{k_1} \rightarrow  V_{k_2} \rightarrow \cdots \rightarrow V_{k_{\gamma}}) \in \text{Fl}(E)$$ to the tuple $$(E_x \rightarrow V_{k_1}, E_x \rightarrow V_{k_2}, \cdots , E_x \rightarrow V_{k_{\gamma}}) \in \prod_{i=1}^{\gamma} \text{Gr}_{r_{k_i}}(E).$$ 
The second embedding is a product of the Pl\"{u}cker embeddings.\\

For every $1 \le i \le \gamma$, define the map 
\begin{eqnarray}\label{phii}
\Phi_i:=\text{Pr}_i \circ \Phi: \text{Fl}(E) \to \text{Gr}_{r_{k_i}}(E),
\end{eqnarray}
where $$\text{Pr}_i : \prod_{i=1}^{\gamma} \text{Gr}_{r_{k_i}}(E) \rightarrow \text{Gr}_{r_{k_i}}(E)$$ 
is the $i$-th projection. Let $\theta_i$
be the degree of the quotient $E/E_{k_i}$. 
Define  $$\omega_i : = \mathcal{O}_{\text{Gr}_{r_{k_i}}(E)}(1) - \theta_i \mathcal{L}_i.$$

Then we have the following results. 

\begin{theorem}\cite[Proposition 4.1]{BP} 
Let $1 \le i \le \gamma$. 
	The divisor $\omega_i$ is nef. Further $\omega_i$ and $\mathcal{L}_i$ generate the nef cone of $\text{Gr}_{r_{k_i}}(E)$.
\end{theorem}

Let $\tilde{\omega_i}$ be the pullback $\Phi_i^*(\omega_i)$ and let $\mathcal{L}=\pi^*(\mathcal{L}')$, where $\mathcal{L'}$ is a degree $1$ line bundle on $X$.

\begin{theorem}\cite[Theorem 5.1]{BP} \label{nef-cone}
	The divisors $\tilde{\omega_1}, \ldots, \tilde{\omega_{\gamma}}, \mathcal{L}$ 
	 generate the nef cone of $\text{Fl}(E)$.
\end{theorem}

\begin{remark}
The description of the nef cone of $\text{Fl}(E)$ given in \cite{BP} is valid in general for any flag bundle over $X$. We consider only flags with dimensions of successive quotients determined by the Harder-Narasimhan filtration of $E$ since this is needed for later computations. 
\end{remark}

\section{Effective cone of curves on $\text{Fl}(E)$}\label{cone-curves}

In this section, we describe some curves in $\text{Fl}(E)$ which are dual to the generators of the nef cone given in Theorem \ref{nef-cone}.  


Let $x \in X$ be an arbitrary point.  
We claim that there exist smooth rational curves $L_1,\ldots, L_{\gamma}$
in $\text{Fl}(E)$  satisfying the following properties:\\

(1) Each $L_i$ are contained in the fiber $\pi^{-1}(x)=\text{Fl}(E_x)$.

(2) $L_i$ are lines in the following sense: for every $1 \le i \le \gamma$, the image $\Phi_i(L_i)$ inside the fiber $\text{Gr}_{r_{k_i}}(E_x)$ is a line with respect to the Pl\"{u}cker embedding.

(3) $\Phi_j(L_i)$  is a point for every $j \ne i$. 

The construction is as follows. 
Let $n$ denote the rank of $E$. 
Fix a $\complex$--basis $\{e_1, e_2, \cdots, e_n\}$ of the fiber $E_x$. 
For every  $1\le j \le \gamma$, set $s_{k_j}:= n - r_{k_{\gamma -j+1}}$. 

Now let $1\le i \le \gamma$.  For every 
$[t_1: t_2] \in \mathbb{P}^1$, 
define a flag of subspaces of $E_x$ 
\begin{equation}\label{flag}
E_x \supset V_{k_1} \supset V_{k_2} \supset \cdots \supset V_{k_i}[t_1, t_2] \supset \cdots \supset V_{k_{\gamma}}
\end{equation}
as follows: 
\begin{eqnarray*}
	V_{k_j}&:=&\mathbb{C}\langle e_1, e_2, \cdots, e_{s_{k_j}}\rangle,  \text{ for }  j < i\\
	V_{k_i}[t_1, t_2] &:=& \mathbb{C}\langle e_1 + e_2, e_2 + e_3, \cdots, e_{s_{k_i}-1} + e_{s_{k_i}}, t_1 e_{s_{k_i}} + t_2 e_{s_{{k_i}+1}}\rangle\\
	V_{k_j} &:=& \mathbb{C}\langle e_1 + e_2, e_2 + e_3, \cdots, e_{s_{k_{j}}} + e_{s_{k_j}+1}\rangle,  \text{ for } j >i
\end{eqnarray*}

Note that $\text{dim}(V_{k_j}) = s_{k_j}= n - r_{k_{\gamma -j +1}}$ for all $1\le j\le \gamma$. 

The set of all flags of subspaces of $E_x$ as in \eqref{flag}
will be denoted by $L_i$ and it is a subset of  $\text{Fl}(E_x)$, and hence of $\text{Fl}(E)$.

From the construction it follows that $\Phi_j(L_i)$ is the constant point $E_x \supset V_{k_j}$ for $j \ne i$. Let $\Phi_i(L_i)$ be the image in $\text{Gr}_{r_{k_i}}(E_x)$. It follows from the construction that the image of $\Phi_i(L_i)$ under the Pl\"{u}cker embedding inside $\mathbb{P}(\wedge^{r_{k_i}} E_x)$ has dimension one and it is defined by linear homogeneous polynomials. Hence the image is a line which gives a variety structure on $L_i$ as a subvariety of $\text{Fl}(E)$.    

\begin{remark}
	The lines $L_i$ in the above construction can be defined functorially. The part of the subspace $V_{k_i}[t_1, t_2]$ that depends on $[t_1:t_2]$ is the line
 \[
 \{t_1 e_{s_{k_i}} + t_2 e_{s_{k_i}+1}\} \subseteq \mathbb{C}\langle e_{s_{k_i}}, e_{s_{k_i}+1}\rangle. 
 \]
 
 We have an isomorphism $\mathbb{C}\langle e_{s_{k_i}}, e_{s_{k_i}+1}\rangle  \cong \mathbb{C}^2$ by sending 
 $$e_{s_{k_i}} \mapsto (1, 0), \; \;\;\; e_{s_{k_{i+1}}} \mapsto (0, 1).$$ 
 
 We have the family of lines passing through the origin parameterized by $\mathbb{P}^1$ 
 \[
 \{t_1 \cdot (1, 0) + t_2 \cdot (0, 1)\} \subseteq \mathbb{C}^2.
 \]
 More explicitly, the family is the following
 \[
 \{((x, y), (t_1, t_2)) : t_1 \cdot y = t_2 \cdot x\} \subseteq \mathbb{C}^2 \times \mathbb{P}^1.
 \]
 
 This is precisely the tautological line bundle $\mathcal{O}_{\mathbb{P}^1}(-1) \subseteq \mathcal{O}_{\mathbb{P}^1}^{\oplus 2}$. Thus the subspaces $V_{k_i}$ will define the  vector bundle $\mathcal{O}_{\mathbb{P}^1} \oplus \mathcal{O}_{\mathbb{P}^1} \oplus \cdots \oplus \mathcal{O}_{\mathbb{P}^1} \oplus \mathcal{O}_{\mathbb{P}^1}(-1)$ of rank $s_{k_i}$ over $\mathbb{P}^1$. We have the following filtration of vector bundles over $\mathbb{P}^1$:
 \[
 \mathcal{O}_{\mathbb{P}^1}^{\oplus n} \supset \mathcal{O}_{\mathbb{P}^1}^{\oplus s_{k_1}} \supset \cdots \supset \mathcal{O}_{\mathbb{P}^1}^{\oplus s_{k_i} -1 } \oplus \mathcal{O}_{\mathbb{P}^1}(-1) \supset \mathcal{O}_{\mathbb{P}^1}^{\oplus s_{k_{i+1}}} \supset \cdots \supset \mathcal{O}_{\mathbb{P}^1}^{\oplus s_{k_{\gamma}}}, 
 \]
 where the embedding $\mathcal{O}_{\mathbb{P}^1}^{\oplus s_{k_{i+1}}} \subset \mathcal{O}_{\mathbb{P}^1}^{\oplus s_{k_i} -1 } \oplus \mathcal{O}_{\mathbb{P}^1}(-1)$ is given by $(\text{id}, 0)$. 
 
 The embedding $\mathcal{O}_{\mathbb{P}^1}^{\oplus s_{k_i} -1 } \oplus \mathcal{O}_{\mathbb{P}^1}(-1) \subset \mathcal{O}_{\mathbb{P}^1}^{\oplus s_{k_{i-1}} }$ is the following natural embedding 
 \[
 \mathbb{P}^1 \times \big(\mathbb{C}\langle e_1 + e_2, e_2 + e_3, \cdots, e_{s_{k_i} -1} + e_{s_{k_i}}\rangle \big) \oplus \mathbb{C}\langle t_1 e_{s_{k_i}} + t_2 e_{s_{k_i+1}}\rangle  \subseteq \mathbb{P}^1 \times \big(\mathbb{C}\langle e_1, e_2, \cdots, e_{s_{k_{i-1}}}\rangle \big),
 \]
 where $s_{k_{i-1}} > s_{k_i}$. The above filtration will define a unique map
 \[
 \mathbb{P}^1 \rightarrow \text{Fl}(E_x)
 \]
 and the image of this map is exactly $L_i$ which was defined above. 
\end{remark}
We recall that the Harder-Narasimhan filtration of $E$ is given by
$$0=E_0 \subset E_1 \subset E_2 \subset \cdots \subset E_{d-1} \subset E_d=E.$$ The rank of $E/E_{k_i}$ is $r_{k_i}$. So $E/E_{k_{i+1}}=(E/E_{k_i})/(E_{k_{i+1}}/E_{k_i})$. This gives the following sequence of quotients over $X$: 
\begin{equation}
\label{section}
	E \rightarrow E/E_{k_1} \rightarrow E/E_{k_2} \rightarrow \cdots \rightarrow E/E_{k_{\gamma}}.
\end{equation} 

This defines a section $s:X \rightarrow \text{Fl}(E)$.

\begin{proposition}\label{curves-cone}
	The curves $L_1, \ldots, L_{\gamma}, s(X)$ generate the closed cone of curves $\overline{\text{NE}}(\text{Fl}(E))$ of Fl$(E)$.
\end{proposition}

\begin{proof}
	To prove the claim, we will show that the curves $L_1,\ldots, L_{\gamma}, s(X)$ are dual to the generators 
	$\tilde{\omega_1},\ldots, \tilde{\omega}_{\gamma},\mathcal{L}$ of the nef cone of $\text{Fl}(E)$. 
	
	We have the following commutative diagram
	
	\[ \begin{tikzcd}
	L_i \arrow[r, hook, "h"] \arrow[swap]{d}{(\Phi_i)_{res}} & \text{Fl}(E) \arrow{d}{\Phi_i} \\%
	\Phi_i(L_i) \arrow[r, hook, "h'"]& \text{Gr}_{r_{k_i}}(E)
	\end{tikzcd}
	\]

Note that $(\Phi_i)_{res}$ is an isomorphism. Indeed, the following morphism is the inverse of $(\Phi_i)_{res}$ when restricted to $\Phi_i(L_i)$
 \[
 \text{Gr}_{r_{k_i}}(E_x) \rightarrow \prod_{i=1}^{\gamma} \text{Gr}_{r_{k_i}}(E_x)
 \]
 \[
 z  \mapsto \bigg\{[E_x \twoheadrightarrow V_{k_1}], \cdots, [E_x \twoheadrightarrow V_{k_{i-1}}], z, [E_x \twoheadrightarrow V_{k_{i+1}}], \cdots, [E_x \twoheadrightarrow V_{k_{\gamma}}]\bigg\} 
 \]

Hence we have: 

\begin{eqnarray*} L_i \cdot \tilde{\omega_i}&=&L_i \cdot \Phi_i^*(\omega_i)\\&=&\text{deg}\left((\Phi_i \circ h)^*(\omega_i)\right)\\&=& \text{deg}\left((h' \circ (\Phi_i)_{res})^*(\omega_i)\right) \\&=&\text{deg}\left( (h')^*(\omega_i)\right) \\&=&\Phi_i(L_i) \cdot \omega_i\\&=&1.\end{eqnarray*}

Similarly, using the following commutative diagram 
	
	\[ \begin{tikzcd}
	L_i \arrow[r, hook, "h"] \arrow[swap]{d}{(\Phi_{i'})_{res}} & \text{Fl}(E) \arrow{d}{\Phi_{i'}} \\%
	\{\text{pt}\} \arrow[r, hook, "h'"]& \text{Gr}_{r_{k_i}}(E)
	\end{tikzcd}
	\]
	we can see that $L_i \cdot \tilde{\omega_{i'}}=0$ for $i \ne i'$.\\
	
	 Since the curve $L_i$ is contained in a fiber $\pi^{-1}(x)$ for some $x \in X$, the composition map $L_i \hookrightarrow \text{Fl}(E) \rightarrow X$ is a constant map. Thus $L_i \cdot \mathcal{L}=0$ for all $i$. \\
	
	The sequence of quotients $\eqref{section}$ defines the section $s(X)$ and the quotient $E \rightarrow E/E_{k_i}$ defines a section $s'(X)$ of $f_i: \text{Gr}_{r_{k_i}}(E) \rightarrow X$. Thus the image $\Phi_i(s(X))$ is precisely $s'(X)$. So we have the following commutative diagram:
	
		\[ \begin{tikzcd}
	s(X) \arrow[r, hook, "s"] \arrow[swap]{d}{(\Phi_{i})_{res}} & \text{Fl}(E) \arrow{d}{\Phi_{i}} \\%
	s'(X) \arrow[r, hook, "s'"]& \text{Gr}_{r_{k_i}}(E)
	\end{tikzcd}
	\]
	
	Then $$s(X) \cdot \tilde{\omega_i}=\text{deg}\left(s^*  (\Phi_i^*(\omega_i))\right)=(s' \circ (\Phi_{i})_{res})^*(\omega_i)=s'(X) \cdot \omega_i,$$  since $(\Phi_{i})_{res}:s(X) \rightarrow s'(X)$ is an isomorphism. 
	
	So  $$s'(X) \cdot \omega_i=(\mathcal{O}_{\text{Gr}_{r_{k_i}}(E)}(1) - \theta_i \mathcal{L}_i) \cdot s'(X)=\theta_i-\theta_i=0.$$
	
	Let $\pi_{\text{res}}: s(X) \hookrightarrow \text{Fl}(E) \rightarrow X$ be the composition map. Then $\pi_{res}$ is an isomorphism with inverse given by the section map $s:X \rightarrow \text{Fl}(E)$. We then see $$s(X) \cdot \mathcal{L}= \text{deg}(\pi_{res}^*(\mathcal{L'})) = \text{deg}(\mathcal{L'})=1.$$
	 \end{proof}

\section{Seshadri constants of line bundles on $\text{Fl}(E)$} \label{sc}

In this section, we prove our main results on Seshadri constants of nef line bundles on $\text{Fl}(E)$. 

By Proposition \ref{nef-cone}, a nef line bundle $L$ on $\text{Fl}(E)$ is of the form $$a_1 \tilde{\omega_1} + a_2 \tilde{\omega_2} + \cdots + a_{\gamma} \tilde{\omega_{\gamma}} +  b \mathcal{L},$$ for some non-negative integers $a_1,\ldots, a_{\gamma}, b$. We denote the line bundle $L$ simply by the tuple $(a_1, a_2, \cdots, a_{\gamma}, b)$. 

By Proposition \ref{curves-cone}, an effective 
curve $C$ in $\overline{\text{NE}}(\text{Fl}(E))$ is of the form 
$$p_1 L_1 + p_2 L_2 + \cdots + p_{\gamma} L_{\gamma} + r s(X),$$ for 
some non-negative integers $p_1,\ldots, p_{\gamma},r$. We denote the curve $C$ simply 
by the tuple $(p_1, p_2, \cdots, p_{\gamma}, r)$.

\begin{lemma}\label{curves}
	For each point $y$ in $\text{Fl}(E)$ and for every $1 \le i \le \gamma$,
	there exist smooth curves $D_i \subset \text{Fl}(E)$ passing through $y$ which are numerically equivalent to the curve $L_i$.
\end{lemma}

\begin{proof}
	Let $\pi(y) = x$ in $X$. The lines $L_i$ constructed at the beginning of Section \ref{cone-curves} are in the fiber $\text{Fl}(E_x)$ over $x$. 
	
	The group $\text{GL}(E_x)$ acts transitively on $\text{Fl}(E_x)$. Let $L_i[0:1]$ be the point on the line $L_i$  corresponding to $[0:1] \in \mathbb{P}^1$. Let $g$ be in $\text{GL}(E_x)$ such that $g\cdot L_i[0:1] = y$.  
	
	Let $D_i$ be the image $g \cdot L_i$ of $L_i$ under the linear automorphism $g$ of $\text{Fl}(E_x)$. For any divisor $D$ in $\text{Fl}(E)$, we will consider the divisor $D' = D \cdot \text{Fl}(E_x)$ in $\text{Fl}(E_x)$. Since $L_i$ and $D_i$ are isomorphic as subschemes of $\text{Fl}(E_x)$ by a linear automorphism of $\text{Fl}(E_x)$, 
	we have $L_i \cdot D' = D_i \cdot D'$.
\end{proof}

\begin{theorem}\label{main-thm}
Let $X$ be a smooth complex projective curve and let $E$ be a vector bundle on $X$ which is not semistable. Let 
$\text{Fl}(E)$ be the flag bundle as in \eqref{flag-define}. 
	\label{inequalitytheorem}

	Let $L = (a_1, a_2, \cdots, a_{\gamma}, b)$ be a nef line bundle on $\text{Fl}(E)$ expressed in terms of the generators of the nef cone given in Proposition \ref{nef-cone}. Then the Seshadri constants of $L$ at any point 
	$y \in \text{Fl}(E)$ satisfy the following inequalities: 
	\begin{center}
		$\min(a_1, a_2, \cdots, a_{\gamma}, b) \le \varepsilon(L, y) \le \min(a_1, a_2, \cdots, a_{\gamma})$.
	\end{center}
\end{theorem}

\begin{proof} Let $y \in \text{Fl}(E)$.
	By Lemma \ref{curves}, for every $i$ there exist smooth curves $D_i$ passing through $y$ which are numerically equivalent to the curves $L_i$. So for every $1 \le i \le \gamma$, 
	\begin{equation}
		\frac{D_i \cdot L}{\text{mult}_x {D_i}}=a_i.
	\end{equation}
	This implies that $\varepsilon(L, y) \le \min(a_1, a_2, \cdots, a_{\gamma})$, giving one of the required inequalities. 

	Let $C \subset \text{Fl}(E)$ be an irreducible and reduced curve passing through $y$. For the proof of the other inequality, we consider two cases.
	
	\textbf{\underline{Case 1}}: Suppose that $C$ is not contained inside the fiber $\text{Fl}(E_x)$ over the point $x: = \pi(y) \in X$. Then by B\'ezout's theorem, 
	$$C \cdot \text{Fl}(E_x) \ge \text{mult}_y C \cdot \text{mult}_y \text{Fl}(E_x).$$ Since $ \text{Fl}(E_x)$ is a smooth variety, $\text{mult}_y \text{Fl}(E_x) = 1$. The associated line bundle to the divisor 
	$\text{Fl}(E_x) \subset \text{Fl}(E)$ is $\mathcal{L}$. So 
	\begin{eqnarray}\label{mult}
	C \cdot \mathcal{L} \ge \text{mult}_y C.
	\end{eqnarray}
	
	By Proposition  \ref{curves-cone}, $C$ is numerically equivalent to $p_1 L_1 + p_2 L_2 + \cdots + p_{\gamma} L_{\gamma} + r s(X)$ for some non-negative integers $p_1,\ldots,p_{\gamma},r$. Then, by \eqref{mult}, 
	$$r \ge \text{mult}_y C.$$
	
Then
	\begin{eqnarray*}
		\frac{C \cdot L}{\text{mult}_y C}&=&\frac{(p_1 L_1 + p_2 L_2 + \cdots + p_{\gamma} L_{\gamma} + r s(X)) \cdot (a_1 \tilde{\omega_1} + a_2 \tilde{\omega_2} + \cdots + a_{\gamma} \tilde{\omega_{\gamma}} +  b \mathcal{L})}{\text{mult}_y C}\\
		 &\ge& \frac{p_1  a_1 + \ldots+ p_{\gamma}a_{\gamma}+r b}{r} \ge b.
	\end{eqnarray*}
\textbf{\underline{Case 2}}:  Suppose now that the curve $C$ is contained inside the fiber $\text{Fl}(E_x)$ over $x$. Then 
$C$ is numerically equivalent to $p_1 L_1 + p_2 L_2 + \cdots + p_{\gamma} L_{\gamma}$ for 
some non-negative integers $p_1,\ldots,p_{\gamma}$. 

We have the following natural embedding of $\text{Fl}(E_x)$. See \eqref{phii}. 

\begin{center}
	$\Phi=(\Phi_1, \Phi_2, \cdots, \Phi_{\gamma}) : \text{Fl}(E_x) \hookrightarrow \prod_{i=1}^{\gamma} \text{Gr}_{r_{k_i}}(E_x)$
\end{center}

For $1 \le i \le \gamma$, the image $\Phi_i(C)$ of the curve $C$ is contained in $ \text{Gr}_{r_{k_i}}(E_x)$. We will denote the image by $C_i$. It  is numerically equivalent to $p_i \Phi_i(L_i)$, where $\Phi_i(L_i)$ is a line in $\text{Gr}_{r_{k_i}}(E_x)$.

Let $\mathcal{O}_{\text{Gr}_{r_{k_i}}(E_x)}(1)$ be the tautological ample line bundle on the Grassmannian $\text{Gr}_{r_{k_i}}(E_x)$. We will denote it by $\mathcal{O}_i(1)$, for simplicity. Then $\mathcal{O}_1(1) \boxtimes \mathcal{O}_2(1) \boxtimes \cdots \boxtimes \mathcal{O}_{\gamma}(1)$ is an ample line bundle on $\prod_{i=1}^{\gamma} \text{Gr}_{r_{k_i}}(E_x)$. 


Let $\mathcal{O}_{\text{Fl}(E_x)}(1): = \Phi^*(\mathcal{O}_1(1) \boxtimes \mathcal{O}_2(1) \boxtimes \cdots \boxtimes \mathcal{O}_{\gamma}(1))$, which is a very ample line bundle on $\text{Fl}(E_x)$. Choose an effective Cartier divisor $Y$ in the linear system $|\mathcal{O}_{\text{Fl}(E_x)}(1)|$ such that $y \in Y \cap C$ and $C \nsubseteq Y$. By B\'ezout's theorem, $$C \cdot Y \ge (\text{mult}_y C) (\text{mult}_y Y).$$ Thus 
\begin{eqnarray}\label{mult2}
\text{mult}_y C \le  C \cdot Y.
\end{eqnarray}

For $1 \le i \le \gamma$, 
let $H_i$ be a general hyperplane in $\text{Gr}_{r_{k_i}}(E_x)$. Then we may write 
$$Y = Y_1+\ldots+Y_{\gamma},$$ where 
$$Y_1 \in |H_1 \times \text{Gr}_{r_{k_2}}(E_x) \times \cdots \times \text{Gr}_{r_{k_{\gamma}}}(E_x) |, \ldots,$$
$$Y_{\gamma} \in |\text{Gr}_{r_{k_1}}(E_x) \times \text{Gr}_{r_{k_2}}(E_x)  \times \cdots \times \text{Gr}_{r_{k_{\gamma-1}}}(E_x) \times H_{\gamma} |.$$

For $j \ne i \in \{1,\ldots,\gamma\}$, we have $$L_j \cdot \big(\text{Gr}_{r_{k_1}}(E_x) \times \text{Gr}_{r_{k_2}}(E_x) \times \cdots \times H_i \times \cdots \times \text{Gr}_{r_{k_{\gamma}}}(E_x) \big) = 0,$$ because the 
line $L_j$ maps to a constant on $\text{Gr}_{r_{k_i}}(E_x)$. Thus a general hyperplane $H_i$ will not pass through this point. 

On the other hand, for any $i \in \{1,\ldots,\gamma\}$,
$$L_i \cdot \big(\text{Gr}_{r_{k_1}}(E_x) \times \text{Gr}_{r_{k_2}}(E_x) \times \cdots \times H_i \times \cdots \times \text{Gr}_{r_{k_{\gamma}}}(E_x) \big) = 1.$$

Hence 
$$C \cdot Y = (p_1 L_1 + p_2 L_2 + \cdots + p_{\gamma} L_{\gamma}) \cdot Y = 
p_1 + p_2 + \cdots + p_{\gamma}.$$

Using \eqref{mult2}, we obtain 

\begin{eqnarray*}
	\frac{C \cdot L}{\text{mult}_y C} &\ge& \frac{C \cdot L}{C\cdot Y} \\
	&=&\frac{a_1p_1 + a_2p_2 + \cdots + a_{\gamma}p_{\gamma}}{p_1 + p_2 + \cdots + p_{\gamma}} \\ &\ge& \min(a_1, a_2, \cdots, a_{\gamma}).
\end{eqnarray*}

Combining the Cases 1 and 2, we conclude that every Seshadri ratio of $L$ at $y$ is 
at least $b$ (happens in Case 1) or 
at least $\min(a_1, a_2, \cdots, a_{\gamma})$ (happens in Case 2).

Thus  $\varepsilon(L, x) \ge \min(a_1, a_2, \cdots, a_{\gamma}, b)$.
This completes the proof of Theorem.
\end{proof}

Theorem \ref{main-thm} immediately gives the following. 

\begin{corollary}
Let the notation be as in Theorem \ref{main-thm}. Let $L = (a_1, a_2, \cdots, a_{\gamma}, b)$ be a nef line bundle on $\text{Fl}(E)$ and suppose that $b \ge \min(a_1, a_2, \cdots, a_{\gamma})$. 
Then $\varepsilon(L, y) = \min(a_1, a_2, \cdots, a_{\gamma})$ for all $y \in \text{Fl}(E)$. 
\end{corollary}





We now show that the lower bound in Theorem \ref{main-thm} is achieved for some points in $\text{Fl}(E)$.

\begin{proposition}\label{min}
	At a point $y$ in the section $s(X)$, $\varepsilon(L, y)$ will achieve the minimum value i.e.,  $\varepsilon(L, y) = \min(a_1, a_2, \cdots, a_{\gamma}, b)$. 
\end{proposition}

\begin{proof}
	Since we have $\varepsilon( L, y) \le \min(a_1, a_2, \cdots, a_{\gamma})$, it is enough to show that $\varepsilon(L , y) \le b$ for any $y \in \text{Fl}(E)$.
	
	We see easily that the Seshadri ratio corresponding to the curve $s(X)$ is precisely $b$. Indeed, 
	note that $s(X)$ is smooth since it is isomorphic to $X$. So 
$$\frac{s(X) \cdot L}{ \text{mult}_y s(X)} = \frac{s(X) \cdot L}{ 1 } = s(X) \cdot (a_1 \tilde{\omega_1} + a_2 \tilde{\omega_2} + \cdots + a_{\gamma} \tilde{\omega_{\gamma}} +  b \mathcal{L}) = b.$$

This show that $\varepsilon(L, y) \le b$ and proves the corollary. 
\end{proof}

\subsection{Seshadri constants at general points of $\text{Fl}(E)$} In this section, we show that the upper bound 
in Theorem \ref{main-thm} is achieved for a general point of $\text{Fl}(E)$ under an additional assumption. 

We quickly recall notation from Section \ref{intro}. See \eqref{flag-define}. 

Let $$0=E_0 \subset E_1 \subset E_2 \subset \cdots \subset E_{d-1} \subset E_d=E$$ be the Harder-Narasimhan filtration of $E$. Fix $1 \le k_1 < k_2 < \cdots <k_{\gamma} \le d-1$.

Then the flag bundle $\text{Fl}(E) = \text{Flag}(r_{k_1}, r_{k_2}, \cdots , r_{k_{\gamma}}, E) $ parametrizes flags of fibers 
of the form 

$$E_x \rightarrow V_{k_1} \rightarrow  V_{k_2} \rightarrow \cdots \rightarrow V_{k_{\gamma}},$$

where $\text{dim}(V_{k_i})=r_{k_i}: = \text{rank } E/E_{k_i}$,  for $1 \le i \le \gamma$.


We now make the following additional assumption. 

\begin{assumption}\label{assume}~
\begin{enumerate}
\item For each $1\le i \le \gamma$, there exists $c_i \in \{1,\ldots, d\}$ such that  $\rank(E_{c_i}) = r_{k_i}$.
\item For  $1\le i \le \gamma$, let  $\zeta_i:  = \text{deg} (E_{c_i})$.  Then  
$\zeta_i$ is an integer multiple of $r_{k_i}$.
\end{enumerate}
\end{assumption}

For the reminder of this section, we assume that Assumption \ref{assume} holds. 


Next two results will be used for computing Seshadri constants at general points of $\text{Fl}(E)$. 

\begin{proposition}\cite[Theorem 4.1]{BHP}\label{bhp}
	The pseudo-effective cone of $\text{Gr}_{r_{k_i}}(E)$ is generated by $\mathcal{O}_{\text{Gr}_{r_{k_i}}(E)}(1) - \zeta_i \mathcal{L}_i$ and $\mathcal{L}_i$.
\end{proposition}

\begin{proposition}\cite[Proposition 2.3]{BHNN}
	The divisor $\mathcal{O}_{\text{Gr}_{r_{k_i}}(E)}(1) - \zeta_i \mathcal{L}_i$ is effective and there exists a unique effective divisor in the linear system $|\mathcal{O}_{\text{Gr}_{r_{k_i}}(E)}(1) - \zeta_i \mathcal{L}_i|$, i.e., 
	$$\dim \text{H}^0(\mathcal{O}_{\text{Gr}_{r_{k_i}}(E)}(1) - \zeta_i \mathcal{L}_i) =1 .$$
\end{proposition}

The map $\Phi_i$ in the following commutative diagram is surjective. 
\[ \begin{tikzcd}
\text{Fl}(E) \arrow[r,hook, "\Phi"] \arrow[dr,"\Phi_i"]
&  \prod_{i=1}^{\gamma} \text{Gr}_{r_{k_i}}(E) \arrow[d, "Pr_i"]\\%
& \text{Gr}_{r_{k_i}}(E)
\end{tikzcd}
\]

Let $D_i$ be the effective Cartier divisor on $\text{Gr}_{r_{k_i}}$
corresponding to the line bundle $\mathcal{O}_{\text{Gr}_{r_{k_i}}(E)}(1) - \zeta_i \mathcal{L}_i$. Then the pullback 
$\Phi_i^{-1} D_i$ is an effective Cartier divisor in $| \Phi_i^{*} (\text{Gr}_{r_{k_i}}(E)(1) - \zeta_i \mathcal{L}_i)|$. By commutativity, we have $\Phi_i^{-1} D_i = \Phi^{-1} ({Pr_i}^{-1}(D_i))$. The pullback of the divisor under the projection $Pr_i$ is the effective divisor $$\text{Gr}_{r_{k_1}}(E) \times \text{Gr}_{r_{k_2}}(E) \times \cdots \times D_i \times \cdots \times \text{Gr}_{r_{k_{\gamma}}}(E).$$ Thus the effective divisor $\Phi_i^{-1} D_i$ in $\text{Fl}(E)$ is $$\text{Fl}(E) \cap \big(\text{Gr}_{r_{k_1}}(E) \times \text{Gr}_{r_{k_2}}(E) \times \cdots \times D_i \times \cdots \times \text{Gr}_{r_{k_{\gamma}}}(E)\big).$$ 

\begin{remark}
	Note that dim $H^0(Pr_i^{*} (\text{Gr}_{r_{k_i}}(E)(1) - \zeta_i \mathcal{L}_i)) = \text{dim } H^0(\text{Gr}_{r_{k_i}}(E)(1) - \zeta_i \mathcal{L}_i) = 1$. We consider the restriction map of sections 
	\begin{center}
	$H^0(Pr_i^{*} (\text{Gr}_{r_{k_i}}(E)(1) - \zeta_i \mathcal{L}_i)) \rightarrow H^0({Pr_i^{*} (\text{Gr}_{r_{k_i}}(E)(1) - \zeta_i \mathcal{L}_i)}|_{\text{Fl}(E)})$ 
	\end{center}
This map is injective. If the only section $s$ unique upto scaling maps to zero section, then $\text{Fl}(E)$ will be contained inside the divisor $\text{Gr}_{r_{k_1}}(E) \times \text{Gr}_{r_{k_2}}(E) \times \cdots \times D_i \times \cdots \times \text{Gr}_{r_{k_{\gamma}}}(E)$ which is not possible. The restriction of the section $s$ is same as pulling back the section via the map $\Phi_i$. So pullback of this section defines the above effective Cartier divisor in $\text{Fl}(E)$.
\end{remark}

Define a closed subvariety  $\mathcal{S} \subset \text{Fl}(E)$ as follows: 
\begin{eqnarray*}
\mathcal{S}  &=& \cap_{i=1}^{\gamma} \big(\text{Fl}(E) \cap \big(\text{Gr}_{r_{k_1}}(E) \times \text{Gr}_{r_{k_2}}(E) \times \cdots \times D_i \times \cdots \times \text{Gr}_{r_{k_{\gamma}}}(E)\big)\big) \\ 
&=& \text{Fl}(E) \cap \big( \cap_{i=1}^{\gamma} \text{Gr}_{r_{k_1}}(E) \times \text{Gr}_{r_{k_2}}(E) \times \cdots \times D_i \times \cdots \times \text{Gr}_{r_{k_{\gamma}}}(E) \big) \\
&=&\text{Fl}(E) \cap D_1 \times D_2 \times \cdots \times D_{\gamma}.
\end{eqnarray*}
Note that $\mathcal{S}$ has  codimension $\gamma$ in $\text{Fl}(E)$.

\begin{theorem}\label{general-point}
Let the notation be as in Theorem \ref{main-thm}. Assume that Assumption \ref{assume} holds. Let $L = (a_1, a_2, \cdots, a_{\gamma}, b)$ be a nef line bundle on $\text{Fl}(E)$. 
	If $y \notin \mathcal{S}$, then $\varepsilon(L, y) = \min(a_1, a_2, \cdots, a_{\gamma})$. 
\end{theorem}

\begin{proof}
	Since the point $y$ does not belong to $\mathcal{S}$, let us assume without loss of generality that 
	$$y \notin D_1 \times \text{Gr}_{r_{k_2}}(E) \times \text{Gr}_{r_{k_3}}(E) \times \cdots \times \text{Gr}_{r_{k_{\gamma}}}(E).$$ 
	
	
	 Let $C$ be a curve in $\text{Fl}(E)$ passing through $y$. Suppose that
	  $C$ is numerically equivalent to $p_1 L_1 + p_2 L_2 + \cdots + p_{\gamma} L_{\gamma} + r s(X)$, for some non-negative integers $p_1,\ldots, p_{\gamma}, r$. 
	 
	 If $C$ is contained inside a fiber of $\pi: \text{Fl}(E) \rightarrow X$, then by $\textbf{Case 2}$ of Theorem \ref{main-thm}, 
	 $$\frac{C \cdot L}{\text{mult}_y C} \ge \min(a_1, a_2, \cdots, a_{\gamma}).$$
	
	 We assume now that $C$ is not contained inside a fiber. By B\'ezout's theorem we have $C \cdot \text{Fl}(E_x) \ge \text{mult}_y C$. Hence  $$C \cdot \text{Fl}(E_x) = r \ge \text{mult}_y C.$$
	 
	 Since $C$ contains $y$, $C$ is not contained in $D_1 \times \text{Gr}_{r_{k_2}}(E) \times \text{Gr}_{r_{k_3}}(E) \times \cdots \times \text{Gr}_{r_{k_{\gamma}}}(E).$
	  
	 
	 Let $C_1 \subset \text{Gr}_{r_{k_1}}(E)$ be the projection of $C$ under $\Phi_1$. Then $C_1$ is numerically equivalent to $p_1 \Phi_1(L_1) + r s(X)$, where $\Phi_1(L_1)$ is a line in $\text{Gr}_{r_{k_1}}(E)$. The image of the section $s(X)$ under the map $\Phi_1$ is the section in $\text{Gr}_{r_{1_{\gamma}}}(E)$ defined by the quotient $E \rightarrow E/E_{k_1}$ and this image is isomorphic to $s(X)$ . 
	Note that $$\theta_1: = \text{deg }   E/E_{k_1} = s(X) \cdot \mathcal{O}_{\text{Gr}_{r_{k_1}}(E)}(1).$$
	 Since $C_1$ is not contained in the effective Cartier divisor $D_1$, $C_1 \cdot D_1 \ge 0$.  So
	 
	 \begin{center}
	 	$(p_1 \Phi_1(L_1) + r s(X))\cdot (\mathcal{O}_{\text{Gr}_{r_{k_1}}(E)}(1) - \zeta_1 \mathcal{L}_1) = p_1 + r \theta_1 - r \zeta_1 = p_1 + r(\theta_1 - \zeta_1) \ge 0.$
	 \end{center}
Then $$p_1 + r(\theta_1 - \zeta_1) \ge 0 \Rightarrow p_1 \ge r(\zeta_1 - \theta_1).$$
 
 Since $\omega_1 = \mathcal{O}_{\text{Gr}_{r_{k_1}}(E)}(1) - \theta_1 \mathcal{L}_1$ is nef, $\omega_1$ is pseudo-effective. By Theorem \ref{bhp}, we can then write 
 
 \begin{center}
 	$\mathcal{O}_{\text{Gr}_{r_{k_1}}(E)}(1) - \theta_1 \mathcal{L}_1 = (\mathcal{O}_{\text{Gr}_{r_{k_1}}(E)}(1) - \zeta_1 \mathcal{L}_1) + m \mathcal{L}$, for an integer $m\ge 0$.
	 \end{center}  
Hence $m = \zeta_1 - \theta_1 \ge 0$. If $\zeta_1 - \theta_1 = 0$, then the nef cone and the pseudo-effective 
cone of $\text{Gr}_{r_{k_1}}(E)$ are the same which implies that the vector bundle $E$ is semistable, by \cite[Corollary 4.3]{BHP}. This contradicts our assumption. Thus $\zeta_1 - \theta_1 \ge 1$. Thus $p_1 \ge r(\zeta_1 - \theta_1) \ge r$. Then

\begin{eqnarray*}
	\frac{C \cdot L}{\text{mult}_y C} &=& 
	\frac{a_1p_1 + \cdots +a_{\gamma}p_{\gamma} + br}{\text{mult}_y C}\\ 
	&\ge& \frac{a_1 p_1 + \cdots +a_{\gamma}p_{\gamma} + br}{r} \\
	&\ge& \frac{a_1r + \cdots +a_{\gamma}p_{\gamma} + br}{r}\\
	&\ge& a_1 \\
	&\ge& \min(a_1, a_2, \cdots, a_{\gamma}).
\end{eqnarray*}

The proof is now complete by Theorem \ref{inequalitytheorem}. 
\end{proof}

Main results of this section are summarized in the following. 

\begin{corollary}
Let the notation be as in Theorem \ref{main-thm}. Let $L = (a_1, a_2, \cdots, a_{\gamma}, b)$ be a nef line bundle on $\text{Fl}(E)$. Then the following statements hold. 
\begin{enumerate}
\item $\min(a_1, a_2, \cdots, a_{\gamma}, b) \le \varepsilon(L, y) \le \min(a_1, a_2, \cdots, a_{\gamma})$ for all $y \in \text{Fl}(E)$. 
\item $\varepsilon(L) = \min(a_1, a_2, \cdots, a_{\gamma},b)$.
\item Assume that Assumption \ref{assume} holds. Then $\varepsilon(L,1) = \min(a_1, a_2, \cdots, a_{\gamma})$.
\end{enumerate}

\end{corollary}

\begin{question}
Does Theorem \ref{general-point} hold without Assumption \ref{assume}? 
\end{question} 

\subsection{Examples}\label{examples}
In this concluding section, we give some examples where our results apply. In all the examples, $X$ denotes a 
connected smooth complex  projective curve.

\begin{example}
	Let $E = L_1 \oplus L_2 \oplus {\mathcal{O}_X}^{\oplus 3}$ be a vector bundle of rank $5$ over $X$, where $\text{deg}(L_1) = 1$, and $\text{deg}(L_2) =2$. Note that $\mu(E) = \frac{3}{5}$ and $E$ is not semistable. 
	
	The Harder-Narasimhan filtration of $E$ is the following
	\[
	0 \subseteq L_2 \subseteq L_2 \oplus L_1 \subseteq E.
	\]   
	
	We consider the flag bundle $\text{Fl}(4, 3, E)$. The Picard rank of  $\text{Fl}(4, 3, E)$ is $3$. Theorem \ref{main-thm}  gives bounds on Seshadri constants for any nef line bundle $L$ on $\text{Fl}(4, 3, E)$. Further, Proposition \ref{min} also gives the value of $\varepsilon(L)$. However Assumption \ref{assume} does not hold in this case. 
\end{example}

\begin{example}
	Let $E = L_1 \oplus L_2 \oplus {\mathcal{O}_X}^{\oplus 3}$ be a rank $5$ vector bundle on $X$, where $\text{deg}(L_1) = 1$, and $\text{deg}(L_2) = -1$. Then $\mu(E) = 0$ and $E$ is not semistable. 
	
	Then the Harder-Narasimhan filtration of $E$ is the following
	\[
	0 \subseteq L_1 \subseteq L_1 \oplus {\mathcal{O}_X}^{\oplus 3} \subseteq E. 
	\]   
	Consider the flag bundle $\text{Fl}(4, 1, E)$. The Picard rank of  $\text{Fl}(4, 1, E)$ is $3$. 
	
	As above both \ref{main-thm} and Proposition \ref{min} apply, but not Theorem \ref{general-point}.
	
\end{example}

\begin{example}
	Let $E = L_1 \oplus {\mathcal{O}_X}^{\oplus 3} \oplus L_2$ be a vector bundle of rank $5$ over $X$, where $\text{deg}(L_1) =4$, and $\text{deg}(L_2) = -1$. Here $\mu(E) = \frac{3}{5}$ and $E$ is not semistable.

	The Harder-Narasimhan filtration of $E$ is the following
	\[
	0 \subseteq L_1 \subseteq L_1 \oplus {\mathcal{O}_X}^{\oplus 3} \subseteq E. 
	\]   
	We consider the flag bundle  $\text{Fl}(4, 1, E)$. The Picard rank of  $\text{Fl}(4, 1, E)$ is $3$. In this examples, all our results apply:  \ref{main-thm},  Proposition \ref{min} and Theorem \ref{general-point}. 
	
\end{example}

\begin{example}
	Let $E = L_1 \oplus L_2 \oplus L_3 \oplus L_4 \oplus {\mathcal{O}_X}^{\oplus 3} $ be a vector bundle of rank $7$ over $X$, where $\text{deg}(L_1) =3$, $\text{deg}(L_2) =1$, $\text{deg}(L_3) =-1$, and $\text{deg}(L_4) =-2$. 
	Here $\mu(E) = \frac{1}{7}$ and $E$ is not semistable. 
	
	The Harder-Narasimhan filtration of $E$ is the following
	\[
	0 \subseteq L_1 \subseteq L_1 \oplus L_2 \subseteq L_1 \oplus L_2 \oplus {\mathcal{O}_X}^{\oplus 3} \subseteq L_1 \oplus L_2 \oplus L_3 \oplus {\mathcal{O}_X}^{\oplus 3} \subseteq E. 
	\]   
	
	In this case, we can consider several flag bundles which satisfy our set-up.

	\begin{itemize}
		\item All our results apply for the flag bundle $\text{Fl}(2, 1, E)$. 
		
		 \item Some other flag bundles we may consider are given below. Assumption \ref{assume} does not apply to any of them. 
		 
		 $\text{Fl}(5, 1, E)$, $\text{Fl}(5, 2, E)$, $\text{Fl}(6, 5, E)$, $\text{Fl}(6, 2, E)$, $\text{Fl}(6, 1, E)$, $\text{Fl}(5, 2, 1, E)$, $\text{Fl}(6, 2, 1, E)$, or $\text{Fl}(6, 5, 2, 1, E)$.	\end{itemize}

\end{example}

\begin{example}
Let $E = L_1 \oplus L_2 \oplus L_3 \oplus L_4 \oplus {\mathcal{O}_X}^{\oplus 3} $ be a vector bundle of rank $7$ over $X$, where $\text{deg}(L_1) =8$, $\text{deg}(L_2) =2$, $\text{deg}(L_3) =-4$, and $\text{deg}(L_4) =-5$. 
	Here $\mu(E) = \frac{1}{7}$ and $E$ is not semistable.

The Harder-Narasimhan filtration of $E$ is the following
\[
0 \subseteq L_1 \subseteq L_1 \oplus L_2 \subseteq L_1 \oplus L_2 \oplus {\mathcal{O}_X}^{\oplus 3} \subseteq L_1 \oplus L_2 \oplus L_3 \oplus {\mathcal{O}_X}^{\oplus 3} \subseteq E.
\]   
Assumption \ref{assume} holds for the flag bundle $\text{Fl}(6, 5, 2, 1, E)$. So all our results are applicable in this example. 
\end{example}

\bibliographystyle{plain}

\end{document}